\documentclass{amsart}
\usepackage[dvipdfmx]{graphicx}
\usepackage{easywncy}
\newtheorem{thm}{Theorem}[section]
\newtheorem{lem}[thm]{Lemma}
\newtheorem{prop}[thm]{Proposition}

\theoremstyle{definition}
\newtheorem{defn}[thm]{Definition}
\newtheorem{ex}[thm]{Example}
\theoremstyle{remark}
\newtheorem{rem}[thm]{Remark}
\newcommand{\A}{\mathcal{A}}
\newcommand{\HH}{\mathfrak{H}}
\newcommand{\PP}{\mathcal{P}}
\newcommand{\Q}{\mathbb{Q}}

\newcommand{\Z}{\mathbb{Z}}
\newcommand{\sym}{\mathfrak{S}}
\newcommand{\sha}{\mathbin{\widetilde{\mathcyr{sh}}}}
\newcommand{\bigsha}{\mathop{\widetilde{\raisebox{-3pt}{\scalebox{2.5}{$\mathcyr{sh}$}}}}}
\newcommand{\bast}{\mathbin{\bar{*}}}
\DeclareMathOperator{\Bij}{Bij}
\newcommand{\ve}[1]{\boldsymbol{#1}}
\begin{document}
\title{Bowman-Bradley type theorem for finite multiple zeta values}
\author{Shingo Saito}
\address{Faculty of Arts and Science, Kyushu University,
 744, Motooka, Nishi-ku, Fukuoka, 819-0395, Japan}
\email{ssaito@artsci.kyushu-u.ac.jp}
\author{Noriko Wakabayashi}
\address{College of Science and Engineering, Ritsumeikan University,
 1--1--1, Nojihigashi, Kusatsu, Shiga, 525-8577, Japan}
\email{noriko-w@fc.ritsumei.ac.jp}
\keywords{Finite multiple zeta value, Bowman-Bradley theorem}
\subjclass[2010]{Primary 11M32; Secondary 05A19}
\begin{abstract}
 The multiple zeta values are multivariate generalizations of the values
 of the Riemann zeta function at positive integers.
 The Bowman-Bradley theorem asserts that the multiple zeta values
 at the sequences obtained by inserting a fixed number of twos
 between $3,1,\dots,3,1$ add up to a rational multiple of a power of $\pi$.
 We show that an analogous theorem holds in a very strong sense
 for finite multiple zeta values,
 which have been investigated by Hoffman and Zhao among others and recently recast by Zagier.
\end{abstract}
\maketitle

\section{Introduction}
\subsection{Finite multiple zeta values}
The \emph{multiple zeta values} and \emph{multiple zeta-star values}
are real numbers defined by
\begin{align*}
 \zeta(k_1,\dots,k_n)
 &=\sum_{m_1>\dots>m_n\ge1}\frac{1}{m_1^{k_1}\dotsm m_n^{k_n}},\\
 \zeta^{\star}(k_1,\dots,k_n)
 &=\sum_{m_1\ge\dots\ge m_n\ge1}\frac{1}{m_1^{k_1}\dotsm m_n^{k_n}}
\end{align*}
for positive integers $k_1,\dots,k_n$ with $k_1\ge2$.
They are generalizations of the values of the Riemann zeta function
at positive integers and are known to be related to
number theory, algebraic geometry, combinatorics, knot theory,
and quantum field theory among others.
Research on these numbers has mainly been focused
on their numerous linear and algebraic relations;
see for example \cite{Hoffman_survey,Zudilin}
and the references therein for an introduction.

Hoffman~\cite{Hoffman} and Zhao~\cite{Zhao} among others
developed a theory of modulo $p$ values,
for primes $p$, of the finite truncations of the above-mentioned infinite sums,
where the indices of summation are all restricted to be less than $p$.
Following an idea of Zagier~\cite{KZ}, we look at the truncations
in the ring $\A=(\prod_{p}\Z/p\Z)/(\bigoplus_{p}\Z/p\Z)$,
where $p$ runs over all primes;
the elements of $\A$ are of the form $(a_p)_p$, where $a_p\in\Z/p\Z$,
and two elements $(a_p)$ and $(b_p)$ are identified if and only if
$a_p=b_p$ for all but finitely many primes $p$.
Note that $\A$ is a $\Q$-algebra.
We shall simply write $a_p$ for $(a_p)$ since no confusion is likely.

\begin{defn}
 For positive integers $k_1,\dots,k_n$, we define
 \begin{align*}
  \zeta_{\A}(k_1,\dots,k_n)
  &=\sum_{p>m_1>\dots>m_n\ge1}\frac{1}{m_1^{k_1}\dotsm m_n^{k_n}}\in\A,\\
  \zeta_{\A}^{\star}(k_1,\dots,k_n)
  &=\sum_{p>m_1\ge\dots\ge m_n\ge1}\frac{1}{m_1^{k_1}\dotsm m_n^{k_n}}\in\A,
 \end{align*}
 and call them \emph{finite multiple zeta(-star) values} in this paper.
\end{defn}

The finite multiple zeta(-star) values
are similar to multiple zeta(-star) values in many respects
as we shall see in this paper.
They do, however, have some differences,
of which one of the most striking is the following:
\begin{prop}[{\cite[Theorem~4.3]{Hoffman}}]\label{prop:Riemann}
 We have $\zeta_{\A}(k)=\zeta_{\A}^{\star}(k)=0$ for all positive integers $k$.
\end{prop}

\begin{proof}
 Let $p$ be an arbitrary prime larger than $k+1$.
 Taking a primitive root $a$ modulo $p$, we have
 \[
  \sum_{m=1}^{p-1}\frac{1}{m^k}
  \equiv\sum_{i=0}^{p-2}\frac{1}{a^{ik}}
  \equiv\frac{1-a^{-k(p-1)}}{1-a^{-k}}\equiv0\pmod{p}.
 \]
 Since we have proved that $\sum_{m=1}^{p-1}m^{-k}\equiv0\pmod{p}$
 for all but finitely many primes $p$,
 it follows that $\zeta_{\A}(k)=\zeta_{\A}^{\star}(k)=0$ in $\A$.
\end{proof}

\subsection{Bowman-Bradley theorem}
Bowman and Bradley~\cite{BB} proved that
the multiple zeta values at the sequences obtained by
inserting a fixed number of twos between $3,1,\ldots,3,1$
add up to a rational multiple of a power of $\pi$;
Kondo, Tanaka, and the first author~\cite{KST}
obtained the same result for multiple zeta-star values.
Let $\{a_1,\dots,a_l\}^m$ denote the $m$ times repetition
of the sequence $a_1,\dots,a_l$:
\[
 \{a_1,\dots,a_l\}^m=\underbrace{a_1,\dots,a_l,\dots,a_1,\dots,a_l}_{lm}.
\]
For the empty sequence $\emptyset$,
we conventionally set $\zeta(\emptyset)=\zeta^{\star}(\emptyset)=1$.

\begin{thm}[\cite{BB,KST}]
 For all nonnegative integers $m$ and $n$, we have
 \begin{align*}
  \sum_{\substack{\sum_{i=0}^{2m}n_i=n\\n_0,\ldots,n_{2m}\ge 0}}
  \zeta
  (\{2\}^{n_0},3,\{2\}^{n_1},1,\{2\}^{n_2},\ldots,3,\{2\}^{n_{2m-1}},1,\{2\}^{n_{2m}})
  &\in\Q\pi^{4m+2n},\\
  \sum_{\substack{\sum_{i=0}^{2m}n_i=n\\n_0,\ldots,n_{2m}\ge 0}}
  \zeta^{\star}
  (\{2\}^{n_0},3,\{2\}^{n_1},1,\{2\}^{n_2},\ldots,3,\{2\}^{n_{2m-1}},1,\{2\}^{n_{2m}})
  &\in\Q\pi^{4m+2n}.
 \end{align*}
\end{thm}

The theorem is a common generalization of the previously known results that
\[
 \zeta(\{3,1\}^m),\zeta^{\star}(\{3,1\}^m)\in\Q\pi^{4m},\qquad
 \zeta(\{2\}^n),\zeta^{\star}(\{2\}^n)\in\Q\pi^{2n}
\]
for all nonnegative integers $m$ and $n$.

For finite multiple zeta(-star) values,
Hoffman \cite[Equation (15)]{Hoffman} proved that
\[
 \zeta_{\A}(\{c\}^n)=\zeta_{\A}^{\star}(\{c\}^n)=0
\]
for all positive integers $c$ and $n$,
and Zhao \cite[Theorem~3.18]{Zhao} proved that
\[
 \zeta_{\A}(\{a,b\}^m)=\zeta_{\A}^{\star}(\{a,b\}^m)=0
\]
for all positive integers $m$,
of which the special case $a=3$ and $b=1$ was conjectured by Kaneko \cite{Kaneko}.
Our aim in this paper is to generalise Zhao's results
by giving the following Bowman-Bradley type theorem,
which is a corollary of our main theorem:
\begin{thm}\label{thm:BB_modp}
 If $a$ and $b$ are odd positive integers and
 $c$ is an even positive integer,
 then for all nonnegative integers $m$ and $n$ with $(m,n)\ne(0,0)$,
 we have
 \begin{align*}
  &\sum_{\substack{\sum_{i=0}^{2m}n_i=n\\n_0,\ldots,n_{2m}\ge 0}}
  \zeta_{\A}
  (\{c\}^{n_0},a,\{c\}^{n_1},b,\{c\}^{n_2},\ldots,a,\{c\}^{n_{2m-1}},b,\{c\}^{n_{2m}})\\
  &=\sum_{\substack{\sum_{i=0}^{2m}n_i=n\\n_0,\ldots,n_{2m}\ge 0}}
  \zeta_{\A}^{\star}
  (\{c\}^{n_0},a,\{c\}^{n_1},b,\{c\}^{n_2},\ldots,a,\{c\}^{n_{2m-1}},b,\{c\}^{n_{2m}})\\
  &=0.
 \end{align*}
\end{thm}

Setting $n=0$ in Theorem~\ref{thm:BB_modp} gives Zhao's results.

\subsection{Statement of the main theorem}
To state our main theorem, we find it convenient to use an algebraic setup,
due to Hoffman~\cite{Hoffman_algebra} in the case of $\zeta$ and $\zeta^{\star}$.
Let $\HH^1=\Q\langle z_1,z_2,\dots\rangle$
denote the noncommutative polynomial algebra in countably many variables.
The product~$\sha$ on $\HH^1$, due to Muneta~\cite{Muneta},
is the $\Q$-bilinear map $\sha\colon\HH^1\times\HH^1\to\HH^1$
defined inductively by
\[
 1\sha w=w\sha1=w,\quad
 z_kw\sha z_{k'}w'=z_k(w\sha z_{k'}w')+z_{k'}(z_kw\sha w')
\]
for $w,w'\in\HH^1$ and $k,k'\in\Z_{\ge1}$.

\begin{ex}
 We have
 \[
  z_k\sha z_l=z_kz_l+z_lz_k,\qquad
  z_k\sha z_lz_{l'}=z_kz_lz_{l'}+z_lz_kz_{l'}+z_lz_{l'}z_k
 \]
 for $k,l,l'\in\Z_{\ge1}$.
\end{ex}

Define $\Q$-linear maps $Z_{\A},\bar{Z}_{\A}\colon\HH^1\to\A$ by setting
\begin{gather*}
 Z_{\A}(1)=\bar{Z}_{\A}(1)=1,\\
 Z_{\A}(z_{k_1}\dotsm z_{k_l})=\zeta_{\A}(k_1,\dots,k_l),\quad
 \bar{Z}_{\A}(z_{k_1}\dotsm z_{k_l})=\zeta_{\A}^{\star}(k_1,\dots,k_l).
\end{gather*}

For $(m,n)\in\Z_{\ge0}^2\setminus\{(0,0)\}$,
let $I_{m,n}$ denote the set of all sequences
\[
 \ve{a}=(a_1,\dots,a_m;b_1,\dots,b_m;c_1,\dots,c_n)
\]
where $a_1,\dots,a_m$ and $b_1,\dots,b_m$ are
odd positive integers
and $c_1,\dots,c_n$ are even positive integers.
For $\ve{a}=(a_1,\dots,a_m;b_1,\dots,b_m;c_1,\dots,c_n)\in I_{m,n}$,
set
\begin{align*}
 z_{\ve{a}}&=
 \sum_{\substack{\sigma,\tau\in\sym_m\\\rho\in\sym_n}}
 z_{a_{\sigma(1)}}z_{b_{\tau(1)}}\dotsm z_{a_{\sigma(m)}}z_{b_{\tau(m)}}
 \sha z_{c_{\rho(1)}}\dotsm z_{c_{\rho(n)}}\\
 &=\sum_{\sigma,\tau\in\sym_m}
 z_{a_{\sigma(1)}}z_{b_{\tau(1)}}\dotsm z_{a_{\sigma(m)}}z_{b_{\tau(m)}}
 \sha z_{c_1}\sha\dots\sha z_{c_n}\in\HH^1,
\end{align*}
where $\sym_l$ is the symmetric group of degree $l$.

\begin{thm}[Main theorem]\label{thm:main}
 For all $(m,n)\in\Z_{\ge0}^2\setminus\{(0,0)\}$ and $\ve{a}\in I_{m,n}$,
 we have $Z_{\A}(z_{\ve{a}})=\bar{Z}_{\A}(z_{\ve{a}})=0$.
\end{thm}

\begin{proof}[Proof that Theorem~\ref{thm:main} implies Theorem~\ref{thm:BB_modp}]
 Put $\ve{a}=(a,\dots,a;b,\dots,b;c,\dots,c)\in I_{m,n}$.
 Then since
 \[
  z_{\ve{a}}=\sum_{\substack{\sigma,\tau\in\sym_m\\\rho\in\sym_n}}
  z_az_b\dotsm z_az_b\sha z_c\dotsm z_c
  =m!^2n!(z_az_b)^m\sha z_c^n,
 \]
 Theorem~\ref{thm:main} shows that
 $Z_{\A}\bigl((z_az_b)^m\sha z_c^n\bigr)=\bar{Z}_{\A}\bigl((z_az_b)^m\sha z_c^n\bigr)=0$,
 which is equivalent to Theorem~\ref{thm:BB_modp}.
\end{proof}

\section{Proof of the main theorem}
\subsection{Outline of the proof}
For $(m,n)\in\Z_{\ge0}^2\setminus\{(0,0)\}$,
write $P_{m,n}$ for the statement that
$Z_{\A}(z_{\ve{a}})=\bar{Z}_{\A}(z_{\ve{a}})=0$
for all $\ve{a}\in I_{m,n}$.
Then the main theorem says that
$P_{m,n}$ is true for all $(m,n)\in\Z_{\ge0}^2\setminus\{(0,0)\}$.
Our proof consists of the following four lemmas:

\begin{lem}\label{lem:step0}
 The statement $P_{0,n}$ is true for all positive integers $n$.
\end{lem}

\begin{lem}\label{lem:step1}
 Suppose that $m$ is a positive integer such that $P_{m,0}$ is true.
 Then $P_{m,n}$ is true for all nonnegative integers $n$.
\end{lem}

\begin{lem}\label{lem:step2}
 Suppose that $m$ is a positive integer such that
 $P_{m',n}$ is true whenever $m'$ is a positive integer less than $m$
 and $n$ is a nonnegative integer.
 Then $Z_{\A}(z_{\ve{a}})+\bar{Z}_{\A}(z_{\ve{a}})=0$ for all $\ve{a}\in I_{m,0}$.
\end{lem}

\begin{lem}\label{lem:step3}
 Suppose that $m$ is a positive integer such that
 $P_{m',n}$ is true whenever $m'$ is a positive integer less than $m$
 and $n$ is a nonnegative integer.
 Then $Z_{\A}(z_{\ve{a}})=\bar{Z}_{\A}(z_{\ve{a}})$ for all $\ve{a}\in I_{m,0}$.
\end{lem}

It is easy to see that the lemmas imply the main theorem.
Indeed, $P_{1,0}$ follows from Lemmas \ref{lem:step2} and \ref{lem:step3}
because $m=1$ vacuously satisfies the assumption;
Lemma~\ref{lem:step1} then shows that
$P_{1,n}$ is true for all nonnegative integers $n$;
it follows that $m=2$ satisfies
the assumption of Lemmas \ref{lem:step2} and \ref{lem:step3},
and so $P_{2,0}$ is true;
induction proceeds in this manner.

\subsection{Proof of Lemma~\ref{lem:step0}}
Although Lemma~\ref{lem:step0} is a direct consequence
of \cite[Theorem~4.4]{Hoffman},
we give a rather detailed proof of the lemma for the convenience of the reader,
partly because our notation differs from that of \cite{Hoffman} and
partly because some of the concepts introduced will also be necessary afterwards.

\begin{defn}
 The \emph{harmonic products} $*$ and $\bast$ on $\HH^1$
 are the $\Q$-bilinear maps $*,\bast\colon\HH^1\times\HH^1\to\HH^1$
 defined inductively by
 \begin{align*}
  1*w&=w*1=w,&
  z_kw*z_{k'}w'&=z_k(w*z_{k'}w')+z_{k'}(z_kw*w')+z_{k+k'}(w*w'),\\
  1\bast w&=w\bast 1=w,&
  z_kw\bast z_{k'}w'
  &=z_k(w\bast z_{k'}w')+z_{k'}(z_kw\bast w')-z_{k+k'}(w\bast w')
 \end{align*}
 for $w,w'\in\HH^1$ and $k,k'\in\Z_{\ge1}$.
\end{defn}

\begin{ex}
 We have
 \[
  z_k*z_l=z_kz_l+z_lz_k+z_{k+l},\qquad
  z_k\bast z_l=z_kz_l+z_lz_k-z_{k+l}
 \]
 for $k,l\in\Z_{\ge1}$.
\end{ex}

We remark that $\HH^1$ is a commutative $\Q$-algebra
with either $*$ or $\bast$ as its product.

As illustrated by
\begin{align*}
 Z_{\A}(z_k)Z_{\A}(z_l)&=\zeta_{\A}(k)\zeta_{\A}(l)
 =\Biggl(\sum_{p>m\ge1}\frac{1}{m^k}\Biggr)
  \Biggl(\sum_{p>n\ge1}\frac{1}{n^l}\Biggr)\\
 &=\Biggl(\sum_{p>m>n\ge1}+\sum_{p>n>m\ge1}+\sum_{p>m=n\ge1}\Biggr)
   \frac{1}{m^kn^l}\\
 &=\zeta_{\A}(k,l)+\zeta_{\A}(l,k)+\zeta_{\A}(k+l)
  =Z_{\A}(z_kz_l+z_lz_k+z_{k+l})\\
 &=Z_{\A}(z_k*z_l),
\end{align*}
the harmonic products have been defined so that
$Z_{\A}$ and $\bar{Z}_{\A}$ are respectively a $*$- and $\bast$-homomorphism:
\begin{prop}\label{prop:Z_*-hom}
 The maps $Z_{\A},\bar{Z}_{\A}\colon\HH^1\to\A$ are respectively
 a $*$- and $\bast$-homomorphism, i.e.\
 $Z_{\A}(w*w')=Z_{\A}(w)Z_{\A}(w')$ and
 $\bar{Z}_{\A}(w\bast w')=\bar{Z}_{\A}(w)\bar{Z}_{\A}(w')$
 for all $w,w'\in\HH^1$.
\end{prop}

Recall that a \emph{partition} of a set $X$ is a family of pairwise disjoint
nonempty subsets of $X$ with union $X$.

\begin{prop}[{\cite[Theorem~4.4]{Hoffman}}]\label{prop:permutation}
 Let $k_1,\dots,k_n$ be positive integers.
 Then
 \[
  Z_{\A}(z_{k_1}\sha\dotsb\sha z_{k_n})
  =\bar{Z}_{\A}(z_{k_1}\sha\dotsb\sha z_{k_n})=0.
 \]
\end{prop}

\begin{proof}
 Observe that
 \[
  z_{k_1}*\dots*z_{k_n}
  =\sum_{\text{$\Pi$ is a partition of $\{1,\dots,n\}$}}
   \bigsha_{A\in\Pi}z_{\sum_{i\in A}k_i};
 \]
 apply $Z_{\A}$ and use Propositions~\ref{prop:Riemann} and \ref{prop:Z_*-hom}
 to obtain
 \[
  \sum_{\text{$\Pi$ is a partition of $\{1,\dots,n\}$}}
  Z_{\A}\biggl(\bigsha_{A\in\Pi}z_{\sum_{i\in A}k_i}\biggr)=0.
 \]
 This shows by induction on $n$ that $Z_{\A}(z_{k_1}\sha\dotsb\sha z_{k_n})=0$
 whenever $k_1,\dots,k_n$ are positive integers.
 The other equation $\bar{Z}_{\A}(z_{k_1}\sha\dotsb\sha z_{k_n})=0$
 can be proved in a similar fashion by using $\bast$ instead of $*$.
\end{proof}

\begin{proof}[Proof of Lemma~\ref{lem:step0}]
 Immediate from Proposition~\ref{prop:permutation}.
\end{proof}

\subsection{Proof of Lemma~\ref{lem:step1}}
Before presenting a proof for general $m$,
we look at the simple case of $m=1$.
We prove by induction on $n$ that
$Z_{\A}(z_az_b\sha z_{c_1}\sha\dotsm\sha z_{c_n})=0$
for all $(a;b;c_1,\dots,c_n)\in I_{1,n}$,
assuming the base case $n=0$.
Let $n\ge1$ and suppose that the claim is true if $n$ is smaller.
Let $(a;b;c_1,\dots,c_n)\in I_{1,n}$.
Apply $Z_{\A}$ to the identity
\begin{align*}
 z_az_b*(z_{c_1}\sha\dotsm\sha z_{c_n})
 &=z_az_b\sha z_{c_1}\sha\dotsm\sha z_{c_n}\\
 &\quad+\sum_{j=1}^{n}\Biggl(z_{a+c_j}z_b\sha\bigsha_{k\ne j}z_{c_k}\Biggr)
   +\sum_{j=1}^{n}\Biggl(z_az_{b+c_j}\sha\bigsha_{k\ne j}z_{c_k}\Biggr)\\
 &\quad+\sum_{i\ne j}\Biggl(z_{a+c_i}z_{b+c_j}\sha\bigsha_{k\ne i,j}z_{c_k}\Biggr)
\end{align*}
and use the inductive hypothesis to obtain
\[
 0=Z_{\A}(z_az_b)Z_{\A}(z_{c_1}\sha\dotsm\sha z_{c_n})
 =Z_{\A}(z_az_b\sha z_{c_1}\sha\dotsm\sha z_{c_n});
\]
here the inductive hypothesis applies
because adding an even integer does not change parity.
The key to the proof for general $m$ given below is to
find a generalization of the above identity for $m\ge2$.

\begin{proof}[Proof of Lemma~\ref{lem:step1}]
 We prove $P_{m,n}$ by induction on $n$, assuming the base case $n=0$.
 Let $n\ge1$ and assume $P_{m,n'}$ for all integers $n'$ with $0\le n'<n$.
 We only prove that $Z_{\A}(z_{\ve{a}})=0$
 for all $\ve{a}=(a_1,\dots,a_m;b_1,\dots,b_m;c_1,\dots,c_n)\in I_{m,n}$,
 because $\bar{Z}_{\A}(z_{\ve{a}})=0$ can be proved in a similar fashion.

 Let $G$ be a spanning subgraph,
 with all degrees at most $1$, of the complete bipartite graph
 on the vertex set
 $\{a_1,b_1,\dots,a_m,b_m\}\cup\{c_1,\dots,c_n\}$;
 the $2m+n$ vertices are regarded as distinct even if
 some of them are equal as integers.
 Define $a_i'=a_i$ if the vertex $a_i$ is isolated;
 $a_i'=a_i+c_k$ if the vertices $a_i$ and $c_k$ are adjacent.
 Define $b_j'$ in a similar manner.
 Write $c_1',\dots,c_l'$ for the isolated vertices among $c_1,\dots,c_n$.
 Then we have
 \[
  z_{a_1}z_{b_1}\dotsm z_{a_m}z_{b_m}*
  (z_{c_1}\sha\dotsm\sha z_{c_n})
  =\sum_{G}
   (z_{a_1'}z_{b_1'}\dotsm z_{a_m'}z_{b_m'}\sha z_{c_1'}\sha\dotsm\sha z_{c_l'}),
 \]
 where $G$ runs over all such subgraphs.

 Replacing $a_i$ with $a_{\sigma(i)}$ and $b_j$ with $b_{\tau(j)}$,
 and summing over all $\sigma,\tau\in\sym_m$,
 we obtain
 \[
  z_{(a_1,\dots,a_m;b_1,\dots,b_m;\emptyset)}*(z_{c_1}\sha\dotsm\sha z_{c_n})\\
  =\sum_{G}z_{(a_1',\dots,a_m';b_1',\dots,b_m';c_1',\dots,c_l')}.
 \]

 Let us see what happens when we apply $Z_{\A}$ to this equation.
 The left-hand side is obviously $0$.
 In the right-hand side,
 the graph $G$ with no edge yields $Z_{\A}(z_{\ve{a}})$ and
 all the other terms vanish by the inductive hypothesis
 because $(a_1',\dots,a_m';b_1',\dots,b_m';c_1',\dots,c_l')\in I_{m,l}$ with $l<n$
 when $G$ has at least one edge.
 Hence we conclude that $Z_{\A}(z_{\ve{a}})=0$.
\end{proof}

\subsection{Proof of Lemma~\ref{lem:step2}}
\begin{prop}[{\cite[Theorem~4.5]{Hoffman}}]\label{prop:reversal}
 Let $k_1,\dots,k_n$ be positive integers. Then
 \begin{align*}
  \zeta_{\A}(k_n,\dots,k_1)&=(-1)^{k_1+\dots+k_n}\zeta_{\A}(k_1,\dots,k_n),\\
  \zeta_{\A}^{\star}(k_n,\dots,k_1)
  &=(-1)^{k_1+\dots+k_n}\zeta_{\A}^{\star}(k_1,\dots,k_n).
 \end{align*}
\end{prop}

\begin{proof}
 We have
 \begin{align*}
  \zeta_{\A}(k_n,\dots,k_1)
  &=\sum_{p>m_n>\dots>m_1\ge1}\frac{1}{m_n^{k_n}\dotsm m_1^{k_1}}\\
  &=\sum_{p>\tilde{m}_1>\dots>\tilde{m}_n\ge1}
    \frac{1}{(p-\tilde{m}_n)^{k_n}\dotsm(p-\tilde{m}_1)^{k_1}}\\
  &=(-1)^{k_1+\dots+k_n}\sum_{p>\tilde{m}_1>\dots>\tilde{m}_n\ge1}
    \frac{1}{\tilde{m}_1^{k_1}\dotsm\tilde{m}_n^{k_n}}\\
  &=(-1)^{k_1+\dots+k_n}\zeta_{\A}(k_1,\dots,k_n).
 \end{align*}
 The other equation can be proved in the same manner.
\end{proof}

\begin{defn}
 Define a $\Q$-linear transformation $d\colon\HH^1\to\HH^1$
 by setting $d(1)=1$ and
 \[
  d(z_{k_1}\dotsm z_{k_n})
  =\sum_{m=1}^{n}\sum_{0=i_0<i_1<\dots<i_m=n}
   z_{k_{i_0+1}+\dots+k_{i_1}}\dotsm z_{k_{i_{m-1}+1}+\dots+k_{i_m}}
 \]
 for positive integers $k_1,\dots,k_n$.
\end{defn}

\begin{ex}
 We have $d(z_k)=z_k$ and $d(z_kz_l)=z_kz_l+z_{k+l}$.
\end{ex}

As illustrated by
\begin{align*}
 \bar{Z}_{\A}(z_kz_l)
 &=\zeta_{\A}^{\star}(k,l)=\sum_{p>m\ge n\ge1}\frac{1}{m^kn^l}
 =\Biggl(\sum_{p>m>n\ge1}+\sum_{p>m=n\ge1}\Biggr)\frac{1}{m^kn^l}\\
 &=\zeta_{\A}(k,l)+\zeta_{\A}(k+l)=Z_{\A}(z_kz_l+z_{k+l})=Z_{\A}\bigl(d(z_kz_l)\bigr),
\end{align*}
the transformation $d$ has been defined so that $\bar{Z}_{\A}=Z_{\A}\circ d$:
\begin{prop}
 We have $\bar{Z}_{\A}=Z_{\A}\circ d$, i.e.\
 $\bar{Z}_{\A}(w)=Z_{\A}\bigl(d(w)\bigr)$ for all $w\in\HH^1$.
\end{prop}

\begin{lem}\label{lem:semi-reversal}
 Let $k_1,\dots,k_l$ be positive integers, where $l\ge1$.
 Then
 \[
  \sum_{j=0}^{l}(-1)^jd(z_{k_j}\dotsm z_{k_1})* z_{k_{j+1}}\dotsm z_{k_l}=0.
 \]
\end{lem}

\begin{proof}
 The lemma is proved in \cite[Proposition~7.1]{Kawashima};
 it also follows from \cite[Proposition~6]{IKOO},
 where our $d$ is denoted by $S$ and the coefficient $(-1)^j$ is missing.
\end{proof}

\begin{rem}
 When $l=0$, the left-hand side of the equation in Lemma~\ref{lem:semi-reversal}
 should naturally be interpreted as $1$ rather than $0$,
 hence the odd-looking assumption that $l\ge1$.
\end{rem}

For $k\in\Z_{\ge0}$, we write $[k]=\{i\in\Z\mid 1\le i\le k\}$.
For sets $X$ and $Y$ of the same cardinality,
we write $\Bij(X,Y)$ for the set of all bijections from $X$ to $Y$.

\begin{proof}[Proof of Lemma~\ref{lem:step2}]
 Let $\ve{a}=(a_1,\dots,a_m;b_1,\dots,b_m;\emptyset)\in I_{m,0}$.
 Then for each $(\sigma,\tau)\in\sym_m^2$,
 applying Lemma~\ref{lem:semi-reversal} to $l=2m$ and
 $(k_1,\dots,k_l)=(a_{\sigma(1)},b_{\tau(1)},\dots,a_{\sigma(m)},b_{\tau(m)})$
 gives
 \begin{align*}
  &\sum_{i=0}^{m}
   d(z_{b_{\tau(i)}}z_{a_{\sigma(i)}}\dotsm z_{b_{\tau(1)}}z_{a_{\sigma(1)}})
  *z_{a_{\sigma(i+1)}}z_{b_{\tau(i+1)}}\dotsm z_{a_{\sigma(m)}}z_{b_{\tau(m)}}\\
  &-\sum_{i=1}^{m}
   d(z_{a_{\sigma(i)}}z_{b_{\tau(i-1)}}z_{a_{\sigma(i-1)}}
     \dotsm z_{b_{\tau(1)}}z_{a_{\sigma(1)}})
  *z_{b_{\tau(i)}}z_{a_{\sigma(i+1)}}z_{b_{\tau(i+1)}}
   \dotsm z_{a_{\sigma(m)}}z_{b_{\tau(m)}}\\
  &=0.
 \end{align*}
 By summing over all $(\sigma,\tau)\in\sym_m^2$ and applying $Z_{\A}$, we obtain
 \begin{align*}
  &\sum_{i=0}^{m}\sum_{\sigma,\tau\in\sym_m}
  \zeta_{\A}^{\star}(b_{\tau(i)},a_{\sigma(i)},\dots,b_{\tau(1)},a_{\sigma(1)})
  \zeta_{\A}(a_{\sigma(i+1)},b_{\tau(i+1)},\dots,a_{\sigma(m)},b_{\tau(m)})\\
  &-\sum_{i=1}^{m}\sum_{\sigma,\tau\in\sym_m}
  \zeta_{\A}^{\star}
  (a_{\sigma(i)},b_{\tau(i-1)},a_{\sigma(i-1)},\dots,b_{\tau(1)},a_{\sigma(1)})
  \zeta_{\A}
  (b_{\tau(i)},a_{\sigma(i+1)},b_{\tau(i+1)},\dots,a_{\sigma(m)},b_{\tau(m)})\\
  &=0.
 \end{align*}
 For simplicity, we write the left-hand side as
 $\sum_{i=0}^{m}P_i-\sum_{i=1}^{m}Q_i$.
 Since $P_0=Z_{\A}(z_{\ve{a}})$ and $P_m=\bar{Z}_{\A}(z_{\ve{a}})$
 by Proposition~\ref{prop:reversal},
 it suffices to show that
 $P_i=0$ for $i=1,\dots,m-1$ and
 $Q_i=0$ for $i=1,\dots,m$.

 For $i=1,\dots,m-1$, we have
 \begin{align*}
  P_i
  &=\sum_{\sigma,\tau\in\sym_m}
  \zeta_{\A}^{\star}(b_{\tau(i)},a_{\sigma(i)},\dots,b_{\tau(1)},a_{\sigma(1)})
  \zeta_{\A}(a_{\sigma(i+1)},b_{\tau(i+1)},\dots,a_{\sigma(m)},b_{\tau(m)})\\
  &=\sum_{\substack{A,B\subset[m]\\\#A=\#B=i}}
    \Biggl(\sum_{\substack{\sigma'\in\Bij([i],A)\\\tau'\in\Bij([i],B)}}
    \zeta_{\A}^{\star}(b_{\tau'(i)},a_{\sigma'(i)},\dots,b_{\tau'(1)},a_{\sigma'(1)})\Biggr)\\
  &\qquad\times\Biggl(
   \sum_{\substack{\sigma''\in\Bij([m-i],[m]\setminus A)\\\tau''\in\Bij([m-i],[m]\setminus B)}}
   \zeta_{\A}(a_{\sigma''(1)},b_{\tau''(1)},\dots,a_{\sigma''(m-i)},b_{\tau''(m-i)})\Biggr)\\
  &=0
 \end{align*}
 by the hypothesis.
 In a similar fashion, for $i=1,\dots,m$, we have
 \begin{align*}
  Q_i
  &=\sum_{\sigma,\tau\in\sym_m}
  \zeta_{\A}^{\star}(a_{\sigma(i)},b_{\tau(i-1)},a_{\sigma(i-1)},\dots,b_{\tau(1)},a_{\sigma(1)})
  \zeta_{\A}(b_{\tau(i)},a_{\sigma(i+1)},b_{\tau(i+1)},\dots,a_{\sigma(m)},b_{\tau(m)})\\
  &=\sum_{\substack{A,B\subset[m]\\\#A=i\\\#B=i-1}}
    \Biggl(\sum_{\substack{\sigma'\in\Bij([i],A)\\\tau'\in\Bij([i-1],B)}}
    \zeta_{\A}^{\star}(a_{\sigma'(i)},b_{\tau'(i-1)},a_{\sigma'(i-1)},\dots,b_{\tau'(1)},a_{\sigma'(1)})\Biggr)\\
  &\qquad\times\Biggl(
   \sum_{\substack{\sigma''\in\Bij([m-i],[m]\setminus A)\\\tau''\in\Bij([m-i+1],[m]\setminus B)}}
   \zeta_{\A}(b_{\tau''(1)},a_{\sigma''(1)},b_{\tau''(2)},\dots,a_{\sigma''(m-i)},b_{\tau''(m-i+1)})\Biggr)\\
  &=0
 \end{align*}
 because of Proposition~\ref{prop:reversal} and the assumption that $a_1,\dots,a_m,b_1,\dots,b_m$ are all odd,
 and the proof is complete.
\end{proof}

\subsection{Proof of Lemma~\ref{lem:step3}}
Let $a_1,\dots,a_m$ and $b_1,\dots,b_m$ be positive integers, and
write $X$ for the multiset consisting of the $2m$ positive integers.
For $P\subset X$, 
denote by $s(P)$ the sum of the elements of $P$;
denote by $\mu_a(P)$ and $\mu_b(P)$ the numbers of $a$'s and $b$'s
contained in $P$ respectively;
define $\lvert P\rvert=\mu_a(P)!\mu_b(P)!$.

Write $\PP$ for the set of all partitions $\Pi$ of $X$
such that $\lvert\mu_a(P)-\mu_b(P)\rvert\le1$ for every $P\in\Pi$.
For $\Pi\in\PP$, write
$\Pi=\{A_1,\dots,A_k,B_1,\dots,B_k,C_1,\dots,C_l\}$
where $\mu_a(A_i)-\mu_b(A_i)=1$,
$\mu_a(B_i)-\mu_b(B_i)=-1$, and $\mu_a(C_j)=\mu_b(C_j)$,
and define
\[
 z_{\Pi}
 =\Biggl(\prod_{P\in\Pi}\lvert P\rvert\Biggr)
 \sum_{\sigma,\tau\in\sym_k}
 z_{s(A_{\sigma(1)})}z_{s(B_{\tau(1)})}\dotsm z_{s(A_{\sigma(k)})}z_{s(B_{\tau(k)})}
 \sha z_{s(C_1)}\sha\dotsm\sha z_{s(C_l)}.
\]

\begin{ex}
 If $m=1$, then $\PP$ consists of the following two elements:
 \begin{itemize}
  \item $\Pi_1$ consisting of $C_1=\{a_1,b_1\}$,
  for which $z_{\Pi_1}=z_{a_1+b_1}$;
  \item $\Pi_2$ consisting of $A_1=\{a_1\}$ and $B_1=\{b_1\}$,
  for which $z_{\Pi_2}=z_{a_1}z_{b_1}$.
 \end{itemize}
 We thus have
 \[
  \sum_{\Pi\in\PP}z_{\Pi}=z_{a_1}z_{b_1}+z_{a_1+b_1}=d(z_{a_1}z_{b_1}).
 \]

 If $m=2$, then $\PP$ consists of the following 12 elements:
 \begin{itemize}
  \item $\Pi_1$ consisting of $C_1=\{a_1,b_1,a_2,b_2\}$,
  for which $z_{\Pi_1}=4z_{a_1+b_1+a_2+b_2}$;
  \item $\Pi_2$ consisting of $A_1=\{a_1,b_1,a_2\}$ and $B_1=\{b_2\}$,
  for which $z_{\Pi_2}=2z_{a_1+b_1+a_2}z_{b_2}$;
  \item $\Pi_3$ consisting of $A_1=\{a_1,b_2,a_2\}$ and $B_1=\{b_1\}$,
  for which $z_{\Pi_3}=2z_{a_1+b_2+a_2}z_{b_1}$;
  \item $\Pi_4$ consisting of $A_1=\{a_1\}$ and $B_1=\{b_1,a_2,b_2\}$,
  for which $z_{\Pi_4}=2z_{a_1}z_{b_1+a_2+b_2}$;
  \item $\Pi_5$ consisting of $A_1=\{a_2\}$ and $B_1=\{b_1,a_1,b_2\}$,
  for which $z_{\Pi_5}=2z_{a_2}z_{b_1+a_1+b_2}$;
  \item $\Pi_6$ consisting of $C_1=\{a_1,b_1\}$ and $C_2=\{a_2,b_2\}$,
  for which $z_{\Pi_6}=z_{a_1+b_1}\sha z_{a_2+b_2}$;
  \item $\Pi_7$ consisting of $C_1=\{a_1,b_2\}$ and $C_2=\{a_2,b_1\}$,
  for which $z_{\Pi_7}=z_{a_1+b_2}\sha z_{a_2+b_1}$;
  \item $\Pi_8$ consisting of $A_1=\{a_1\}$, $B_1=\{b_1\}$, and $C_1=\{a_2,b_2\}$,
  for which $z_{\Pi_8}=z_{a_1}z_{b_1}\sha z_{a_2+b_2}$;
  \item $\Pi_9$ consisting of $A_1=\{a_1\}$, $B_1=\{b_2\}$, and $C_1=\{a_2,b_1\}$,
  for which $z_{\Pi_9}=z_{a_1}z_{b_2}\sha z_{a_2+b_1}$;
  \item $\Pi_{10}$ consisting of $A_1=\{a_2\}$, $B_1=\{b_1\}$, and $C_1=\{a_1,b_2\}$,
  for which $z_{\Pi_{10}}=z_{a_2}z_{b_1}\sha z_{a_1+b_2}$;
  \item $\Pi_{11}$ consisting of $A_1=\{a_2\}$, $B_1=\{b_2\}$, and $C_1=\{a_1,b_1\}$,
  for which $z_{\Pi_{11}}=z_{a_2}z_{b_2}\sha z_{a_1+b_1}$;
  \item $\Pi_{12}$ consisting of $A_1=\{a_1\}$, $A_2=\{a_2\}$,
  $B_1=\{b_1\}$, and $B_2=\{b_2\}$,
  for which $z_{\Pi_{12}}=z_{a_1}z_{b_1}z_{a_2}z_{b_2}+z_{a_1}z_{b_2}z_{a_2}z_{b_1}
  +z_{a_2}z_{b_1}z_{a_1}z_{b_2}+z_{a_2}z_{b_2}z_{a_1}z_{b_1}$.
 \end{itemize}
 We thus have
 \begin{align*}
  z_{\Pi_1}&=\sum_{\sigma,\tau\in\sym_2}
  z_{a_{\sigma(1)}+b_{\tau(1)}+a_{\sigma(2)}+b_{\tau(2)}},\\
  z_{\Pi_2}+z_{\Pi_3}&=\sum_{\sigma,\tau\in\sym_2}
  z_{a_{\sigma(1)}+b_{\tau(1)}+a_{\sigma(2)}}z_{b_{\tau(2)}},\\
  z_{\Pi_4}+z_{\Pi_5}&=\sum_{\sigma,\tau\in\sym_2}
  z_{a_{\sigma(1)}}z_{b_{\tau(1)}+a_{\sigma(2)}+b_{\tau(2)}},\\
  z_{\Pi_6}+z_{\Pi_7}&=\sum_{\sigma,\tau\in\sym_2}
  z_{a_{\sigma(1)}+b_{\tau(1)}}z_{a_{\sigma(2)}+b_{\tau(2)}},\\
  z_{\Pi_8}+\dots+z_{\Pi_{11}}&=\sum_{\sigma,\tau\in\sym_2}
  (z_{a_{\sigma(1)}+b_{\tau(1)}}z_{a_{\sigma(2)}}z_{b_{\tau(2)}}+
   z_{a_{\sigma(1)}}z_{b_{\tau(1)}+a_{\sigma(2)}}z_{b_{\tau(2)}}+
   z_{a_{\sigma(1)}}z_{b_{\tau(1)}}z_{a_{\sigma(2)}+b_{\tau(2)}}),\\
  z_{\Pi_{12}}&=\sum_{\sigma,\tau\in\sym_2}
  z_{a_{\sigma(1)}}z_{b_{\tau(1)}}z_{a_{\sigma(2)}}z_{b_{\tau(2)}}
 \end{align*}
 and so
 \[
  \sum_{\Pi\in\PP}z_{\Pi}=\sum_{\sigma,\tau\in\sym_2}
  d(z_{a_{\sigma(1)}}z_{b_{\tau(1)}}z_{a_{\sigma(2)}}z_{b_{\tau(2)}})
 \]
\end{ex}

\begin{lem}\label{lem:partition}
 We have
 \[
  \sum_{\Pi\in\PP}z_{\Pi}
  =\sum_{\sigma,\tau\in\sym_m}
  d(z_{a_{\sigma(1)}}z_{b_{\tau(1)}}\dotsm z_{a_{\sigma(m)}}z_{b_{\tau(m)}}).
 \]
\end{lem}

\begin{proof}
 Succinctly speaking, the left-hand side is the expansion of the right-hand side.
 To be more precise,
 for each $\Pi\in\PP$,
 each monomial $w$ that appears in the expansion of $z_{\Pi}$
 appears in the right-hand side
 exactly as many times as there are pairs $(\sigma,\tau)\in\sym_m^2$
 for which
 $d(z_{a_{\sigma(1)}}z_{b_{\tau(1)}}\dotsm z_{a_{\sigma(m)}}z_{b_{\tau(m)}})$
 gives rise to the monomial $w$;
 the number of such $\sigma$ is $\prod_{P\in\Pi}\mu_a(P)!$ and
 the number of such $\tau$ is $\prod_{P\in\Pi}\mu_b(P)!$,
 from which it follows that the number of such pairs $(\sigma,\tau)$ is
 \[
  \prod_{P\in\Pi}\mu_a(P)!\cdot\prod_{P\in\Pi}\mu_b(P)!
  =\prod_{P\in\Pi}\lvert P\rvert.
 \]
 This proves the lemma.
\end{proof}

\begin{proof}[Proof of Lemma~\ref{lem:step3}]
 Let $\ve{a}=(a_1,\dots,a_m;b_1,\dots,b_m;\emptyset)\in I_{m,0}$.
 Then Lemma~\ref{lem:partition} shows that
 \[
  \bar{Z}_{\A}(z_{\ve{a}})=Z_{\A}\bigl(d(z_{\ve{a}})\bigr)
  =\sum_{\Pi\in\PP}Z_{\A}(z_{\Pi}).
 \]
 If $\Pi=\bigl\{\{a_1\},\dots,\{a_m\},\{b_1\},\dots,\{b_m\}\bigr\}$,
 then $z_{\Pi}=z_{\ve{a}}$;
 otherwise, $z_{\Pi}$ is an integer multiple of $z_{\ve{b}}$
 for some $\ve{b}\in\bigcup_{1\le m'<m}\bigcup_{n\ge0}I_{m',n}$,
 and so $Z_{\A}(z_{\Pi})=0$ by the hypothesis.
 It follows that $\bar{Z}_{\A}(z_{\ve{a}})=Z_{\A}(z_{\ve{a}})$.
\end{proof}

\section*{Acknowledgements}
The authors would like to thank Masanobu Kaneko and Tatsushi Tanaka
for helpful comments, and Shuji Yamamoto for carefully reading the manuscript
and making suggestions for improving the exposition.


\begin{thebibliography}{99}
 \bibitem{BB}
  D. Bowman and D. M. Bradley,
  \emph{The algebra and combinatorics of shuffles and multiple zeta values},
  J. Combin. Theory Ser. A \textbf{97} (2002), no.~1, 43--61.
 \bibitem{Hoffman_algebra}
  M. E. Hoffman,
  \emph{The algebra of multiple harmonic series},
  J. Algebra \textbf{194} (1997), no.~2, 477--495.
 \bibitem{Hoffman_survey}
  M. E. Hoffman,
  \emph{Algebraic aspects of multiple zeta values},
  Zeta functions, topology and quantum physics, Dev. Math., vol.~14,
  Springer, New York, 2005, pp.~51--73.
 \bibitem{Hoffman}
  M. E. Hoffman,
  \emph{Quasi-symmetric functions and mod $p$ multiple harmonic sums}
  (preprint), math/0401319
 \bibitem{IKOO}
  K. Ihara, J. Kajikawa, Y. Ohno, and J. Okuda,
  \emph{Multiple zeta values vs. multiple zeta-star values},
  J. Algebra \textbf{332} (2011), no.~1, 187--208.
 \bibitem{Kaneko}
  M. Kaneko,
  \emph{Finite multiple zeta values mod $p$ and relations among multiple zeta values},
  S\=urikaisekikenky\=usho K\=oky\=uroku (2012), no.~1813, 27--31,
  Various aspects of multiple zeta values (Japanese) (Kyoto, 2010).
 \bibitem{KZ}
  M. Kaneko and D. Zagier,
  \emph{Finite multiple zeta values},
  in preparation.
 \bibitem{Kawashima}
  G. Kawashima,
  \emph{A class of relations among multiple zeta values},
  J. Number Theory \textbf{129} (2009), no.~4, 755--788.
 \bibitem{KST}
  H. Kondo, S. Saito, and T. Tanaka,
  \emph{The Bowman-Bradley theorem for multiple zeta-star values},
  J. Number Theory \textbf{132} (2012), no.~9, 1984--2002.
 \bibitem{Muneta}
  S. Muneta,
  \emph{A note on evaluations of multiple zeta values},
  Proc. Amer. Math. Soc. \textbf{137} (2009), no.~3, 931--935.
 \bibitem{Zhao}
  J. Zhao,
  \emph{Wolstenholme type theorem for multiple harmonic sums},
  Int. J. Number Theory \textbf{4} (2008), no.~1, 73--106.
 \bibitem{Zudilin}
  V. V. Zudilin,
  \emph{Algebraic relations for multiple zeta values},
  Uspekhi Mat. Nauk \textbf{58} (2003), no.~1 (349), 3--32.
\end{thebibliography}
\end{document}